\documentclass{amsart}
\usepackage{amsmath,amsfonts,euscript,amssymb,amsthm,graphicx,subfigure,verbatim}
\usepackage{hyperref}

\date{\today}

\newtheorem{theorem}{Theorem}

\newtheorem{lemma}{Lemma}
\newtheorem{proposition}{Proposition}

\newcommand{\R}{\mathbb{R}}
\newcommand{\N}{\mathbb{N}}

\newcommand{\St}{\mathbb{S}^2}

\newcommand{\BMO}{\mathrm{BMO}}

\newcommand{\scp}[2]{\left\langle #1, #2 \right\rangle}

\newcommand{\eps}{\varepsilon}
\newcommand{\del}{\partial}

\newcommand{\bs}[1]{ \boldsymbol{#1}}
\newcommand{\hl}{(-\Delta)^{\frac 1 2}}

\newcommand{\hlf}{(-\Delta)^{\frac 1 4}}

\newcommand{\gl}{(-\Delta)}

\def\Xint#1{\mathchoice
{\XXint\displaystyle\textstyle{#1}}%
{\XXint\textstyle\scriptstyle{#1}}%
{\XXint\scriptstyle\scriptscriptstyle{#1}}%
{\XXint\scriptscriptstyle\scriptscriptstyle{#1}}%
\!\int}
\def\XXint#1#2#3{{\setbox0=\hbox{$#1{#2#3}{\int}$}
\vcenter{\hbox{$#2#3$}}\kern-.5\wd0}}

\def\dashint{\Xint-}

\title[Global dissipative half-harmonic flows into spheres]{Global dissipative half-harmonic flows into spheres: small data in critical Sobolev spaces}

\author{Christof Melcher}
\address{RWTH Aachen\\Lehrstuhl I f\"ur Mathematik\\Pontdriesch 14-16\\52056 Aachen}
\address{JARA -- Fundamentals of Future Information Technology}
\email{melcher@rwth-aachen.de}

\author{Zisis N. Sakellaris}
\address{RWTH Aachen\\Lehrstuhl I f\"ur Mathematik\\Pontdriesch 14-16\\52056 Aachen}
\email{sakellaris@math1.rwth-aachen.de}

\begin{document}

\begin{abstract} We establish global existence, uniqueness, regularity and long-time asymptotics of strong solutions
to the half-harmonic heat flow and dissipative Landau-Lifshitz equation, valid for initial data that is small in the homogeneous Sobolev norm
$\dot{H}^{\frac n 2}$ for space dimensions $n \le 3$.

\end{abstract}

\maketitle

\section{Introduction}

For suitable maps $\bs u:\R^n \times \R^+ \rightarrow{\St}\subset{\R^{3}}$ we consider the following fractional version of the dissipative Landau-Lifshitz or Landau-Lifshitz-Gilbert equation
\begin{equation} \label{eq:HLLG}
\del_t \bs u = \bs u \times \hl \bs u + \lambda \bs u \times \bs u \times \hl \bs u
\end{equation}
where $\lambda>0$ is a fixed damping parameter and $\times$ is the vector product in $\R^3$. Compared to conventional 
Landau-Lifshitz equations based on the Dirichlet energy, the geometric Laplacian $-\Delta$
is replaced by the so-called half-Laplacian $\hl$.  It is customary to define  fractional Laplacians $(-\Delta)^{s}$ in Euclidean space via Fourier transform $\mathcal{F}$ by 
$\mathcal{F}\left( (-\Delta)^{s } f \right)(\xi)=|\xi|^{2s} \mathcal{F}(f)(\xi)$.
The half-Laplacian is a nonlocal first order elliptic operator and factorizes as $\hl=\mathcal{R} \cdot \nabla$
where $\mathcal{R}$ is the vectorial Riesz transform acting as a Calder\'on-Zygmund operator. 
In the one-dimensional case, $\mathcal{R}$ becomes the Hilbert transform $\mathcal{H}$.
This factorization proves particularly convenient in combination
with geometric nonlinearities. The half-Laplacian of a field $\bs u$ arises as the $L^2$ gradient of the governing energy
\begin{equation} \label{eq:Henergy} 
E(\bs u)= \frac{1}{2} \| \bs u \|_{\dot{H}^{\frac 1 2}}^2
\end{equation} 
where homogeneous Sobolev norms are given by $\|f\|_{\dot{H}^s}=\|(-\Delta)^{\frac s 2} f\|_{L^2}$. 
The energy \eqref{eq:Henergy} serves as a Liapunov functional for \eqref{eq:HLLG}.
More precisely, \eqref{eq:HLLG} is a dissipative version of the so-called half-wave map equation
\begin{equation} \label{eq:HWM}
\del_t \bs u = \bs u \times \hl \bs u,
\end{equation}
recently examined in \cite{LenSchi:wave, GerLen}, which formally preserves the energy \eqref{eq:Henergy}. In the language of ferromagnetism, \eqref{eq:HWM} is the conservative Landau-Lifshitz equation of a continuous spin system governed by \eqref{eq:Henergy}. Local and global well-posedness questions for \eqref{eq:HWM} have been addressed in \cite{PuGuo2} and \cite{KriegerSire}. The parabolic structure of \eqref{eq:HLLG} is revealed by interpreting the double vector product 
\[
-\bs u \times \bs u \times \bs \xi= (\bs 1 -\Pi_{\bs u} )\bs \xi
\] 
as orthogonal projection onto the tangent planes $T_{\bs u} \St$ along $\bs u$ where 
\[
\Pi_{\bs u} = \bs u \otimes \bs u.
\]
The overdamped limit $\lambda \to \infty$ yields the half-harmonic map heat flow equation
\begin{equation} \label{eq:HHHF}
\del_t \bs u + \hl \bs u = \Pi_{\bs u} \left[  \hl \bs u \right],
\end{equation}
the $L^2$ gradient flow equation for the governing energy \eqref{eq:Henergy}.
In contrast to the dissipative and conservative Landau-Lifshitz equations \eqref{eq:HLLG} and \eqref{eq:HWM}, the heat flow equation \eqref{eq:HHHF} extends to target spheres $\mathbb{S}^{m} \subset \R^{m+1}$ of arbitrary dimension. A new phenomenon of infinite time blow-up of half-harmonic flows with $m=n=1$ has recently been shown in \cite{Sire2017} to occur. This is in sharp contrast to the finite time blow-up of harmonic map heat flows  \cite{Chang92} for $m=n=2$ and is one motivation of this work.\\
Global existence of a certain class of weak solutions of \eqref{eq:HLLG} and \eqref{eq:HHHF} for finite energy initial data has been shown in \cite{PuGuo1, Schikorra_Sire_Wang} for more general governing energies and target manifolds, respectively. Here we are interested in global strong solvability of \eqref{eq:HLLG} and \eqref{eq:HHHF} under appropriate smallness conditions on the initial data. As a measure of smallness, the parabolic dilation symmetry of the Landau-Lifshitz equation predicts the category of scaling invariant norms, independently of the (fractional) power of the Laplacian $(-\Delta)^{s}$. In the conventional case $s=1$, small energy well-posedness in the energy critical dimension $n=2$ is part of results in \cite{Chen2000, DingGuo2000, Harpes} based on techniques developed in the context of the harmonic map heat flow \cite{Riviere_Diss, Struwe, Freire}.
In the supercritical case $n \ge 3$, global well-posedness  under smallness conditions in terms of scaling critical Sobolev norms can be proven by using methods from semilinear evolution equations \cite{Soyeur}. An analog analysis of the dissipative Landau-Lifshitz equation can be carried out on the basis of a moving frame reformulation \cite{Melcher}, see also \cite{LinLiaWang}. Finite time blow-up and existence of only partially regular solutions are to be expected for finite energy initial data, see e.g. \cite{Chen_Struwe, Coron_Ghidaglia, Ding_Wang:07, Melcher:05}.\\
In the half-harmonic case $s=1/2$, the energy critical dimension according to the dilation symmetry of \eqref{eq:HLLG}
and \eqref{eq:HHHF} is $n=1$. Global well-posedness for $n=1$ is therefore expected for small energy \eqref{eq:Henergy} initial data.  In higher dimensions $n > 1$, scaling critical fractional Sobolev spaces
\[
H^{\frac n 2}_Q(\R^n; \mathbb{S}^{m})=\{ \bs u:\R^n \to \mathbb{S}^{m}: \bs u -Q \in H^{\frac n 2}(\R^n;\R^{m+1})\}
\]
and their homogeneous seminorms
are a natural replacement of the energy space for $n=1$. For simplicity, we shall focus on the physical dimensions $n \le 3$. It is conceivable, however, that the result holds true in arbitrary space dimensions.

\begin{theorem} \label{MainTheorem}
Let $Q\in{\St}$ and $\bs u_{0}\in H^{\frac{n}{2}}_{Q}(\R^n;\St)$ for $1 \le n \le 3$. Then there exists $\delta>0$ with the property that if 
$\|\bs u_0\|_{\dot{H}^{\frac{n}{2}}} <  \delta$, there exists
a unique strong solution 
\[
\bs u  \in C^0\left([0,\infty); H^{\frac{n}{2}}_Q(\R^n;\St)\right)
\]
with 
\[
(\nabla \bs u, \del_t \bs u) \in  L^2\left((0,\infty);H^{\frac{n-1}{2}}\left(\R^n; \R^{3 \times (n+1)}\right)\right)
\]
of \eqref{eq:HLLG} with $\bs u(0)=\bs u_0$. The solution
is locally H\"older continuous in space-time and 
\[
\lim_{t \to \infty} \| \nabla \bs u(t)\|_{\dot{H}^{\frac{n-1}{2}}}=0.
\]
In particular, if $n=2,3$ then $\bs u(t)$ converges uniformly to $Q$ as $t \to \infty$.
\end{theorem}

In the case of the half-harmonic heat flow equation \eqref{eq:HHHF}, the result holds true for target spheres $\mathbb{S}^{m} \subset \R^{m+1}$ of arbitrary dimension.
The theorem includes in particular an energy threshold for infinite time blow-up as constructed in
\cite{Sire2017} in the case $m=n=1$. In view of the classification results for half-harmonic maps from $\R$ to $\mathbb{S}^1$
in \cite{MillotSire}, one may conjecture that $\delta=\sqrt{2\pi}$, corresponding to the energy \eqref{eq:Henergy} of a half-harmonic map of unit degree (winding number). The regularity theory of finite energy half-harmonic maps is subtle and based on novel commutator estimates \cite{DaLioRiv1, DaLioRiv2, LenSchi}. These estimates are a replacement for the Jacobian structure of the harmonic map equation, leading to the regularity theory in dimension $n=2$, see \cite{Helein:book}.  
In contrast, the proof of Theorem \ref{MainTheorem} rests on classical analytical tools.

In view of the orthogonality relation $\nabla \bs u \cdot \bs u=0$, the nonlinearity coming from the projection $\Pi_{\bs u}$ can be expressed in terms of a tensorial commutator
\begin{equation} \label{eq:commutator_structure}
\bs u \cdot \hl \bs u = - [\mathcal{R},  \bs u] \nabla \bs u
\end{equation}
where more precisely
\[
 [\mathcal{R}, \bs u] \nabla \bs u = \sum_{j,k} [ \mathcal{R}_j, u_k] \del_j u_k
\]
and $[\mathcal{R}_{j}, b] f = \mathcal{R}_j (bf)-b (\mathcal{R}_j f)$ for a scalar Riesz transform $\mathcal{R}_j$.
Taking into account the Poincar\'e type estimate $\|b\|_{\rm BMO} \lesssim \|b\|_{\dot{H}^{\frac{n}{2}}}$, the commutator structure \eqref{eq:commutator_structure} is exploited
by means of the Coifman-Rochberg-Weiss commutator theorem \cite{CRW} 
\begin{equation} \label{eq:CRW}
\| [\mathcal{R}_j ,b ] f \|_{L^p} \lesssim \|b\|_{\rm BMO} \|f\|_{L^p} \quad \text{for} \quad 1<p<\infty,
\end{equation}
where $ \|b\|_{\rm BMO} = \sup_{Q} \dashint_Q |b- \dashint_Q b|$ bounds the mean oscillation of $b$ over finite cubes $Q \subset \R^n$.
A further nonlinear structure arising from the Hamiltonian terms in \eqref{eq:HLLG} is
\begin{equation} \label{eq:skew_structure}
 \bs u \times \hl \bs v = (-\Delta)^{\frac 1 4} \left( \bs u \times (-\Delta)^{\frac 1 4} \bs v \right) - \left[ (-\Delta)^{\frac 1 4}, \Omega_{\bs u}\right] (-\Delta)^{\frac 1 4} \bs v,
\end{equation}
where $\Omega_{\bs u}$ is the matrix field representing the vector multiplication by $\bs u$, i.e.,
\[
\Omega_{\bs u} \bs \xi = \bs u \times \bs \xi
\] 
and $\bs v$ is typically a (fractional) derivative of $\bs u$.  A tensorial version of the homogenenous Kato-Ponce commutator theorem \cite{KatoPonce, Hofmann}  
 \begin{equation} \label{eq:KatoPonce}
\| [(-\Delta)^{\frac 1 4} ,a ] f \|_{L^q} \lesssim \| (-\Delta)^{\frac 1 4} a \|_{L^r} \|f\|_{L^p} \quad \text{for}  \quad \frac 1 q = \frac 1 p +\frac 1 r
\end{equation}
serves as a generalized Leibniz rule for \eqref{eq:skew_structure}. 
Finally fractional versions of the Gagliardo-Nirenberg inequality
serve as a replacement for the Ladyzhenskaya interpolation inequality, crucial in the analysis of harmonic map heat flow and conventional  Landau-Lifshitz-Gilbert equations in two space dimensions. These are presented in Appendix \ref{appendix}, together with all the other interpolation inequalities used in the course of this work.

The construction of a global solution is based on an $\eps$ regularization of the half-Laplacian 
$\eps (-\Delta)^{\nu} + \hl$ with an integer $\nu$ to be chosen appropriately.
With the notation for the normal projection $\Pi_{\bs u}$ and tangential rotation $\Omega_{\bs u}$ introduced beforehand,
the corresponding regularization of \eqref{eq:HLLG} reads
\begin{equation}\label{eq:LLGR}
\del_t \bs u + \lambda   \left( \eps \gl^\nu  + \hl \right ) \bs u   = \left( \lambda \Pi_{\bs u} + \Omega_{\bs u} \right)\left( \eps \gl^\nu  + \hl \right) \bs u. 
\end{equation}
For the nonlinearities on the right, we shall refer to the projection and Hamiltonian terms, respectively. For $\eps>0$ and $\nu \ge \frac{n+1}{2}$ the
regularized equation becomes subcritical in the sense that the governing energy
\begin{equation} \label{eq:Renergy}
E_\eps(\bs u) = \frac 1 2 \left( \eps \|\nabla^{\nu} \bs u\|_{L^2}^2 + \|\bs u\|_{\dot{H}^{\frac{1}{2}}}^2 \right)
\end{equation}
serves to control higher order Sobolev norms. This entails to global existence and regularity by means of standard arguments as in \cite{Zhou_Guo_Shao} while treating fractional terms a lower order perturbations. With such approximate
solutions for regularized initial data at hand, we shall derive, under the condition of uniformly small $H^{\frac{n}{2}}$ seminorms, 
$\eps$-uniform Sobolev estimates that enable us to pass to the limit $\eps \searrow 0$ producing the desired solution. The uniqueness result serves to bootstrap regularity in space and time. In the presentation we shall focus on the dissipative Landau-Lifshitz equation which is more intricate due to the Hamiltonian terms. The dimensionality of the target sphere will play no role in our analysis of the heat flow terms. For the sake of clarity we shall first treat the energy critical case $n=1$
and the cases $n=2,3$, requiring a biharmonic regularization with $\nu=2$, separately.

\section{Uniqueness}

We establish uniqueness of solutions of \eqref{eq:HLLG} in the class
\begin{equation} \label{eq:class}
X_T=\left\{ \bs u \in C^0\left([0,T]; H^{\frac{n}{2}}_Q(\R^n;\St)\right) : \nabla \bs u, \del_t \bs u \in L^2\left((0,T);H^{\frac{n-1}{2}}\left(\R^n \right)\right) \right\}
\end{equation}
for a finite or infinite time horizon $T>0$. Using commutator and fractional Gagliardo-Nirenberg estimates, the common strategy known from harmonic flows, see e.g. \cite{Struwe}, can be adapted.

\begin{proposition}\label{prop:uniqueness}
Let $\bs u,\bs v \in X_T$ be two solutions of \eqref{eq:HLLG}
such that $\bs u(0)=\bs v(0)$, then $\bs u(t)\equiv{\bs v(t)}$ for all $t\in(0,T)$. 
\end{proposition}

\begin{proof}
The difference $\bs w=\bs u-\bs v$ satisfies 
\begin{equation} \label{eq:w}\begin{aligned}
\partial_{t}\bs w+\lambda\hl \bs w=&\bs w\times{\hl{\bs u}}+\bs v\times{\hl{\bs w}}+\lambda(\bs u \cdot \hl \bs u)\bs w\\
&-\lambda(\bs v \cdot \hl \bs v)\bs v+\lambda(\bs u \cdot \hl \bs u)\bs v.
\end{aligned}
\end{equation}
Taking into account $\bs u, \bs v \in \St$ we have $\bs w\cdot \bs v=-\dfrac{1}{2}|\bs w|^{2}$. The projection term
\[
A:=\bs u \cdot \hl {\bs u}+\bs v \cdot \hl {\bs v},
\]
satisfies the estimate
\begin{equation}\label{eq:AEstimate}
\|A\|_{L^{2n}} \lesssim \|\hl \bs u\|_{L^{2n}} + \|\hl \bs v\|_{L^{2n}}.
\end{equation}
Hence Sobolev embedding $\dot H^{\frac {n-1} {2}} (\R^n)\hookrightarrow L^{2n}(\R^n)$ implies that  $\|A\|_{L^{2n}} \in L^2(0,T)$.
Multiplying \eqref{eq:w} by $\bs w$ and integrating in space and time yields 
$$\dfrac{1}{2}\|\bs w(T)\|_{L^{2}}^{2}+\lambda\int_{0}^{T}\|\bs w\|_{\dot{H}^{\frac{1}{2}}}^{2} dt=\dfrac{\lambda}{2}\int_{0}^{T}\int_{\R^n} A  |\bs w|^{2}dxdt -\int_{0}^{T}\langle \bs w,\hl{\bs w}\times{\bs v}\rangle_{L^{2}}dt.$$
We fix $\rho>0$ and estimate the terms on the right separately. Using \eqref{eq:GN1} and \eqref{eq:AEstimate}, along with H\"older's and Young's inequalities, the projection term satisfies
\begin{eqnarray*}
\left| \int_{0}^{T}\int_{\mathbb{R}^n} A  |\bs w|^{2}dxdt \right| & \lesssim & \int_0^T \|A\|_{L^{2n}} \| \bs w \|_{L^{\frac{4n}{2n-1}}}^2 dt\\
& \lesssim & \left(\int_0^T \|A\|_{L^{2n}}^2 dt\right)^{\frac{1}{2}}\left(\sup_{t \in [0,T]} \|\bs w(t)\|_{L^{2}}^2 + \int_{0}^{T} \|\bs w\|_{\dot{H}^{\frac{1}{2}}}^{2}dt\right)\\
& \le &\rho \left(\sup_{t \in [0,T]} \|\bs w(t)\|_{L^{2}}^2 + \int_{0}^{T} \|\bs w\|_{\dot{H}^{\frac{1}{2}}}^{2}dt\right),
\end{eqnarray*}
if $T=T(\rho)$ is sufficiently small. To obtain a corresponding bound for the Hamiltonian term, we use \eqref{eq:KatoPonce},   \eqref{eq:GN2} and \eqref{eq:GN1}  as follows
\begin{eqnarray*}
 \left| \int_{0}^{T}\langle \bs w,\hl{\bs w}\times{\bs v}\rangle_{L^{2}}dt \right| & \leq&
 \int_{0}^{T} \left| \langle[\hlf,\Omega_{\bs v}]\bs w,\hlf{\bs w} \rangle_{L^{2}} \right| dt\\
&\lesssim&{{\int_{0}^{T}\|\hlf \bs v\|_{L^{4n}}}\|\bs w\|_{L^{\frac{4n}{2n-1}}}\|\bs w\|_{\dot{H}^{\frac{1}{2}}} dt}\\
&\lesssim&{{\int_{0}^{T}}\|\hlf \bs v\|_{\dot{H}^{\frac{n}{2}}}^{\frac{1}{2}}\| \hlf \bs v\|_{\dot{H}^{\frac{n-1}{2}}}^{\frac{1}{2}}\|\bs w\|_{L^{2}}^{\frac{1}{2}}\|\bs w\|_{\dot{H}^{\frac{1}{2}}}^{3/2}dt}\\
&\lesssim&\sup_{t \in [0,T]} \|\bs w(t)\|_{L^{2}}^{\frac{1}{2}} \int_{0}^{T}\|\bs v\|_{\dot{H}^{\frac{n}{2}}}^{\frac{1}{2}}\|\bs v\|_{\dot H^{\frac{n+1}{2}}}^{\frac{1}{2}}\|\bs w\|_{\dot{H}^{\frac{1}{2}}}^{3/2}dt.
\end{eqnarray*}
H\"older's inequality yields
\[
\begin{aligned}
\int_{0}^{T}\|\bs v\|_{\dot{H}^{\frac{n}{2}}}^{\frac{1}{2}} & \|\bs v\|_{\dot{H}^{\frac{n+1}{2}}}^{\frac{1}{2}} \|\bs w\|_{\dot{H}^{\frac{1}{2}}}^{3/2} dt \\
\le & \left( \sup_{t \in [0,T]} \|\bs v(t)\|_{\dot{H}^{\frac{n}{2}}}^2 \int_{0}^{T}\|\bs v\|_{\dot{H}^{\frac{n+1}{2}}}^{2} \, dt \right)^{1/4} \left(\int_{0}^{T} \|\bs w\|_{\dot{H}^{\frac{1}{2}}}^{2}dt \right)^{3/4},
\end{aligned}
\]
and Young's inequality implies
\[
\left| \int_{0}^{T}\langle \bs w,\hl{\bs w}\times{\bs v}\rangle_{L^{2}}dt \right|  \le C \rho \left( 
\sup_{t \in [0,T]} \|\bs w(t)\|_{L^{2}}^2 + \int_{0}^{T} \|\bs w\|_{\dot{H}^{\frac{1}{2}}}^{2}dt \right).
\]
Combining the preceding estimates we have
\[
\dfrac{1}{2}\|\bs w(T)\|_{L^{2}}^{2}+\lambda\int_{0}^{T}\|\bs w\|_{\dot{H}^{\frac{1}{2}}}^{2}dt
\leq C\rho\left(\sup_{t\in{[0,T]}} \|\bs w(t)\|_{L^{2}}^{2} + {\int_{0}^{T}}\|\bs w\|_{\dot{H}^{\frac{1}{2}}}^{2}dt \right)
\]
hence $\bs w\equiv 0$ for $C \rho <1$ and times $t\leq{T(\rho)}$. This implies that the maximal time interval such that $\bs w\equiv{0}$ is relatively open. Since by continuity the time interval is also closed, the claim follows.
\end{proof}

The proof may be suitably modified, in order to prove uniqueness of (weak) solutions $\bs u= \bs u^{(\eps)}$ of \eqref{eq:LLGR} 
for $\eps>0$ assuming 
\[
\bs u \in  C^0([0,T]; H^{\frac{n+\nu}{2}}_Q((\R^n;\St)) \quad \text{and} \quad \nabla \bs u \in L^2((0,T); H^{\frac {n+ \nu} {2}}(\R^n)).
\]
Letting $\alpha = \lambda / (1+\lambda^2)$ and $\beta = \alpha / \lambda$
an algebraically equivalent form of \eqref{eq:LLGR} is
\begin{equation} \label{eq:Gilbert}
\beta \partial_t \bs u =  \bs u \times \left[ \alpha \partial_t \bs u +\left( \eps \gl^\nu + \hl\right)  \bs u \right].
\end{equation}
In the half-harmonic case $\eps=0$, the formulations \eqref{eq:HLLG} and \eqref{eq:Gilbert} are analytically equivalent in the class \eqref{eq:class} of strong solutions.

\section{Estimates in the energy critical case $n=1$}

We first focus on the one dimensional case and derive uniform estimates for solutions of \eqref{eq:LLGR} with $\nu=1$, i.e., 
\begin{equation} \label{eq:LLGR1}
\del_t \bs u + \lambda   \left( \eps \gl  + \hl \right ) \bs u   = \left( \lambda \Pi_{\bs u} + \Omega_{\bs u} \right)\left( \eps \gl  + \hl \right) \bs u
\end{equation}
where
\[
\Pi_{\bs u} \gl \bs u = |\nabla \bs u|^2 \bs u.
\]
Approximating initial data $\bs u(0)=\bs u_0$ appropriately we can assume arbitrary 
Sobolev regularity of global solutions $\bs u=\bs u^{(\eps)}$ of \eqref{eq:LLGR} for $\eps>0$. We shall continue using the symbols $\nabla$ and $\Delta$
for the first and second order spatial derivative $\del_x$ and $\del_x^2$, respectively, and $\mathcal{H}$ for the Hilbert transform so that $\hl=\mathcal{H} \nabla$.

\begin{lemma}\label{lemma:H^1}
Suppose $\eps>0$ and $T>0$. Then with $\alpha=\lambda/(1+\lambda^2)$ 
\begin{equation}\label{LLGR:Identity}
E_\eps(\bs u(T))+ \alpha \int_0^T  \| \del_t \bs u\|_{L^2}^2 \, dt = E_\eps(\bs u_0).
\end{equation}
Moreover there exist $\delta>0$ such that for $\|\bs u_0 \|_{\dot H^{\frac{1}{2}}} \le \delta$
\begin{equation}\label{H1/2 Estimate}
\sup_{t \in [0,T]} \|\bs u(t)\|_{\dot{H}^{\frac{1}{2}}}^2 +   \lambda   \int_0^T \left( \| \bs u\|_{\dot{H}^1}^2 + \eps \| \bs u\|_{\dot{H}^{3/2}}^2 \right) \, dt \le  \|\bs u_0\|_{\dot{H}^{\frac{1}{2}}}^2
\end{equation}
and 
\begin{equation} \label{eq:L^2_estimate}
 \|\bs u(T)-Q\|_{H^{\frac{1}{2}}}^2 \le e^{c_0  T} \|\bs u_0-Q\|_{H^{{1/2}}}^2 
\end{equation}
where $c_0$ is a positive constant that only depends on $E_{\eps}(\bs u_0)$.
\end{lemma}

\begin{proof}
From \eqref{eq:Gilbert}
the energy identity \eqref{LLGR:Identity} is obtained from the multiplier $\bs u \times \del_t \bs u$. 

To prove \eqref{H1/2 Estimate} we use the multiplier $\hl \bs u$ 
for \eqref{eq:LLGR1} taking into account \eqref{eq:commutator_structure}
\begin{align*}
\left[  \frac 1 2 \|\bs u\|_{\dot{H}^{\frac{1}{2}}}^2 \right]_0^T&+\lambda \int_0^T \left( \| \bs u\|_{\dot{H}^1}^2 + \eps  \|\bs u\|_{\dot{H}^{3/2}}^2 \right) dt\leq\lambda\int_0^T  \|  [\mathcal{H}, \bs u] \nabla \bs u \|_{L^2}^{2}dt\\
 &+\lambda \eps\int_0^T\|  [\mathcal{H}, \bs u] \nabla \bs u \|_{L^2} \| \nabla \bs u\|_{L^4}^2 dt 
+\eps\int_0^T\left|\scp{\bs u \times \gl \bs u}{ \hl \bs u}_{L^2}\right| dt.
\end{align*}
Using \eqref{eq:CRW} and \eqref{eq:GN2} the projection terms are bounded by
\begin{align*}
  \|  [\mathcal{H}, \bs u] \nabla \bs u \|_{L^2}^{2} &+\varepsilon \|  [\mathcal{H}, \bs u] \nabla \bs u \|_{L^2} \| \nabla \bs u\|_{L^4}^2  \\
&\lesssim  \|\bs u\|_{\rm BMO}^2  \| \nabla \bs u\|_{L^2}^2   + \eps    \|\bs u\|_{\dot{H}^{\frac{1}{2}}} \|\nabla \bs u \|_{L^2} \| \nabla \bs u\|_{L^4}^2  \\
&\lesssim  \|\bs u\|_{\dot{H}^{\frac{1}{2}}}^2  \| \nabla \bs u\|_{L^2}^2   + \eps    \|\bs u\|_{\dot{H}^{\frac{1}{2}}} \|\nabla \bs u \|_{L^2}^2 \| \bs u\|_{\dot{H}^{3/2}}  \\
& \lesssim \|\bs u\|_{\dot{H}^{\frac{1}{2}}}^2  \left(  \|  \bs u\|_{\dot{H}^1}^2  + \eps   \| \bs u\|_{\dot{H}^{3/2}}^2 \right).
\end{align*}
As for the Hamiltonian terms
\begin{eqnarray*}
\scp{\bs u \times \gl \bs u}{ \hl \bs u}_{L^2} &=& \scp{\bs u \times \nabla \bs u}{\nabla \hl \bs u}_{L^2} \\
&=& \scp{  \hlf(\bs u \times \nabla \bs u)}{ \hlf \nabla  \bs u}_{L^2} \\
&=& \scp{ [ \hlf, \Omega_{\bs u}] \nabla \bs u}{ \hlf \nabla  \bs u}_{L^2},
\end{eqnarray*}
using the skew-symmetry of the triple product.
By \eqref{eq:KatoPonce} and interpolation we obtain 
\begin{eqnarray}\label{estimate:Hamiltonian}
\left|\scp{\bs u \times \gl \bs u}{ \hl \bs u}_{L^2} \right| &\lesssim& \|   [ \hlf, \Omega_{\bs u}] \nabla \bs u \|_{L^2} \|\bs u\|_{\dot{H}^{3/2}} \\ 
&\lesssim& \|  \hlf \bs u  \| _{L^4} \| \nabla \bs u \|_{L^4} \|\bs u\|_{\dot{H}^{3/2}} \nonumber \\
&\lesssim& \|  \hlf \bs u  \| _{L^4} \| \bs u \|_{\dot{H}^1}^{\frac{1}{2}} \|\bs u\|_{\dot{H}^{3/2}}^{3/2} \nonumber \\
&\lesssim& \|  \bs u  \| _{\dot{H}^{\frac{1}{2}}}^{\frac{1}{2}} \| \bs u \|_{\dot{H}^1} \|\bs u\|_{\dot{H}^{3/2}}^{3/2}   \nonumber \\
&\lesssim& \|\bs u\|_{\dot{H}^{\frac{1}{2}}}  \|\bs u\|_{\dot{H}^{3/2}}^{2}.\nonumber 
\end{eqnarray}
Hence there exists a universal constant $c>0$ such that 
\begin{equation}
\left[  \frac 1 2 \|\bs u\|_{\dot{H}^{\frac{1}{2}}}^2 \right]_0^T+\left( \lambda - c \left( \sigma(T) + \lambda \sigma(T)^2 \right) \right) \int_0^T \left( \| \bs u\|_{\dot{H}^1}^2 + \eps  \|\bs u\|_{\dot{H}^{3/2}}^2 \right) dt \le 0
\end{equation}
with the continuous function
\[
\sigma(\tau)=\sup_{t \in [0,\tau]}  \|\bs u\|_{\dot{H}^{\frac{1}{2}}}.
\]
Letting $\delta$ be the positive root of the quadratic equation $\lambda = 2 c \left( \sigma + \lambda \sigma^2 \right)$ for $\sigma$, it follows from a continuity argument that $\sigma(T) \le \delta$, implying \eqref{H1/2 Estimate}. \\

Finally to prove \eqref{eq:L^2_estimate} we use the multiplier $\bs u-Q$ for \eqref{eq:LLGR1} to obtain
\[
 \frac{1}{2}  \|\bs u(T)-Q\|_{L^2}^2  + \lambda  \int_0^T (\eps \|\bs u\|_{\dot{H}^1}^2 + \|\bs u\|_{\dot{H}^{\frac{1}{2}}}^2) \, dt  =  \frac{1}{2}  \|\bs u_0-Q\|_{L^2}^2 + \lambda R(T)+S(T),
\]
where by taking into account that $\bs u \cdot (\bs u -Q) = \frac 1 2 |\bs u-Q|^2$
\[
R(T)= \frac{1}{2} \int_0^T \scp{\eps |\nabla \bs u|^2 - [\mathcal{H}, \bs u] \nabla \bs u}{
 |\bs u-Q|^2}_{L^2} dt
\]
and using $\scp{\bs u \times \Delta \bs u}{\bs u-Q}_{L^2}= 
-\scp{\bs u \times \nabla \bs u}{\nabla \bs u}_{L^2}=0$
\[
S(T)= -\int_0^T \scp{ (\bs u - Q) \times \hl (\bs u -Q)}{  \bs u}_{L^2} \, dt.
\]
Regarding $R(T)$ we use \eqref{eq:GN1} and Young's inequality to bound
\begin{eqnarray*}
R(T) &\lesssim& \int_0^T \left(  \eps \|\nabla \bs u\|_{L^4}^2 + \|\nabla \bs u\|_{L^2} \|\bs u\|_{\dot{H}^{\frac{1}{2}}}  \right) \|\bs u-Q\|_{L^4}^2 \, dt \\
&\lesssim& \int_0^T \left(  \eps \|\bs u\|_{\dot{H}^{1}}   \|\bs u\|_{\dot{H}^{3/2}} + \|\bs u\|_{\dot{H}^{1}}  \|\bs u\|_{\dot{H}^{\frac{1}{2}}} \right)  \|\bs u\|_{\dot{H}^{\frac{1}{2}}}   \|\bs u-Q\|_{L^2}dt \\
&\lesssim&  \sup_{t \in [0,T]} \|\bs u\|_{\dot{H}^{\frac{1}{2}}} \left( \eps \int_0^T  \|\bs u\|_{\dot H^{3/2}}^2 dt  +  \sup_{t\in[0,T]} E_\eps (\bs u)  \int_0^T  \|\bs u-Q\|_{L^2}^2 dt \right)\\
 &+&  \sup_{t \in [0,T]} \|\bs u\|_{\dot{H}^{\frac{1}{2}}}^2  \left(\int_0^T  \|\bs u\|_{\dot H^1}^2 dt  + \int_0^T  \|\bs u-Q\|_{L^2}^2 dt  \right)\, .
\end{eqnarray*}
Regarding $S(T)$ we use \eqref{eq:KatoPonce}, \eqref{eq:GN1} and Sobolev interpolation to bound
\begin{eqnarray*}
\left|\scp{ (\bs u - Q) \times \hl (\bs u -Q)}{  \bs u}_{L^2} \right| &\lesssim& \|   [ \hlf, \Omega_{\bs u}]   (\bs u-Q) \|_{L^2} \|\bs u\|_{\dot{H}^{\frac{1}{2}}} \\ 
&\lesssim& \|  \hlf \bs u  \| _{L^4} \|  \bs u -Q \|_{L^4} \|\bs u\|_{\dot{H}^{\frac{1}{2}}} \\
&\lesssim& \|  \hlf \bs u  \| _{L^4} \| \bs u -Q \|_{ L ^2}^{\frac{1}{2}} \|\bs u\|_{\dot{H}^{\frac{1}{2}}}^{3/2}\\
&\lesssim& \|  \bs u  \| _{\dot{H}^{\frac{1}{2}}}^{\frac{1}{2}}   \|  \bs u  \| _{\dot{H}^{1}}^{\frac{1}{2}}  \| \bs u -Q \|_{L^2}^{\frac{1}{2}} \|\bs u\|_{\dot{H}^{\frac{1}{2}}}^{3/2}\\
&\lesssim& \|  \bs u  \| _{\dot{H}^{1}}  \| \bs u -Q \|_{L^2} \|\bs u\|_{\dot{H}^{\frac{1}{2}}}. 
\end{eqnarray*}
Hence it follows from Young's inequality that
\[
 S(T) \lesssim  \sup_{t\in[0,T]} \| \bs u\|_{\dot H^{\frac{1}{2}}}  \left(   \int_0^T   \|\bs u\|_{\dot H^1}^2 dt +   \int_0^T  \|\bs u - Q\|_{L^{2}}^2  dt \right).
\]
The higher order terms on the right are compensated by adding \eqref{H1/2 Estimate}, and the claim follows from Gronwall's inequality.
\end{proof}

Higher order estimates follow by exploiting the commutator structure once more. 

\begin{lemma} \label{lemma:H^{3/2}}
There exists $\delta>0$, such that for $\|\bs u_0\|_{\dot{H}^{\frac{1}{2}}} \le \delta$ and $0< T< \infty$
\[
 \|\bs u(T)\|_{\dot{H}^{1}}^2  +  \lambda  \int_{0}^T \left( \eps \| \bs u\|_{\dot{H}^2}^2+ \|  \bs u\|_{\dot{H}^{3/2}}^2 \right) dt \le  \|\bs u_0\|_{\dot{H}^{1}}^2.
\]
\end{lemma}

\begin{proof}
Using the multiplier $\gl \bs u$ for \eqref{eq:LLGR1} we obtain
\begin{align*}
  \left[  \frac 1 2 \|\bs u\|_{\dot{H}^{1}}^2 \right]_0^T+ \lambda \int_0^T \|  \bs u\|_{\dot{H}^{3/2}}^2dt  & + \eps \int_0^T \|\bs u\|_{\dot{H}^{2}}^2dt \\
&\lesssim \lambda\int_0^T \|  [\mathcal{H}, \bs u] \nabla \bs u \|_{L^2} \| \nabla \bs u\|_{L^4}^2+\eps   \|  \nabla \bs u\|_{L^4}^4 dt\\
& +\int_0^T \|   [ \hlf, \Omega_{\bs u}] \nabla \bs u \|_{L^2} \|\bs u\|_{\dot{H}^{3/2}}dt.
\end{align*}
We claim that terms on the right can be absorbed. Indeed,  
for the $\lambda$ terms we obtain by using \eqref{eq:CRW} and \eqref{eq:GN1}
\begin{align*}
&\|  [\mathcal{H}, \bs u] \nabla \bs u \|_{L^2} \| \nabla \bs u\|_{L^4}^2
+\eps  \|  \nabla \bs u\|_{L^4}^4 \\
&\lesssim  \left( \|\bs u\|_{\rm BMO}  \| \nabla \bs u\|_{L^2} \right) \left( \| \nabla \bs u\|_{\dot H^{{1/2}}}  \| \nabla \bs  u\|_{ L^2}\right) + \eps   \|\nabla \bs u\|_{L^2}^{3}\|\Delta \bs u\|_{L^2}   \\
&\lesssim   \|\bs u\|_{\dot{H}^{\frac{1}{2}}}  \| \nabla \bs u\|_{L^2}^2 \| \bs u\|_{\dot{H}^{3/2}} +
\eps   \|\bs u\|_{\dot{H}^{\frac{1}{2}}}^2 \| \bs u\|_{\dot{H}^{2}}^2 \\
&\lesssim     \|\bs u\|_{\dot{H}^{\frac{1}{2}}}^2   \| \bs u\|_{\dot{H}^{3/2}}^2 +\eps   \|\bs u\|_{\dot{H}^{\frac{1}{2}}}^2 \| \bs u\|_{\dot{H}^{2}}^2 .
\end{align*}
Finally the Hamiltonian terms satisfy
 $$ \|   [ \hlf, \Omega_{\bs u}] \nabla \bs u \|_{L^2} \|\bs u\|_{\dot{H}^{3/2}}  \lesssim  \|\bs u\|_{\dot{H}^{\frac{1}{2}}} \| \bs u\|_{\dot{H}^{3/2}}^2,$$ via estimate \eqref{estimate:Hamiltonian}. 
The claim follows by taking into account the $\dot{H}^{\frac 1 2 }$ estimate in Lemma \ref{lemma:H^1}.
\end{proof}

\section{Estimates in the supercritical case $n=2,3$}

The supercritical case requires higher order extensions of Lemmata \ref{lemma:H^1} and \ref{lemma:H^{3/2}} under the relevant higher order regularization. 
We consider the cases $n=2,3$ and global regular solutions $\bs u = \bs u^{(\eps)}$ to \eqref{eq:LLGR} with $\nu=2$, i.e., 
\begin{equation}\label{LLGR2} 
\del_t \bs u + \lambda   \left( \eps \Delta^2  + \hl \right ) \bs u   = \left( \lambda \Pi_{\bs u} + \Omega_{\bs u} \right)\left( \eps \Delta^2  + \hl \right) \bs u
\end{equation}
with governing energy 
\[
E_{\eps}(\bs u) = \frac 1 2 \left( \eps \| \bs u\|_{\dot{H}^2}^2 + \|\bs u\|_{\dot{H}^{\frac{1}{2}}}^2\right).
\]

The biharmonic regularization introduces slight difficulties through the higher complexity of geometric projection terms
$\Pi_{\bs u}( \Delta^2 \bs u) = (\bs u \cdot \Delta^2 \bs u) \bs u$ where
\begin{eqnarray}\label{eq:constraint} 
 - \bs u\cdot \Delta^2 \bs u &=&   \Delta | \nabla  \bs u|^2 + |\Delta \bs u|^2 + 2\nabla \bs u\cdot\nabla\Delta \bs u\\
 &=& |\Delta \bs u|^2 + 2 |\nabla \nabla \bs u|^2+ 4\nabla \bs u\cdot\nabla\Delta \bs u.  \nonumber
\end{eqnarray}
On the other hand we are now having Sobolev inequalities \eqref{eq:Sob1}, \eqref{eq:Sob2} at our disposal, by which 
we can in particular bypass the Gronwall argument and obtain
an $L^2$ bound that is global in time.

\begin{lemma} \label{lemma:L^2_higher}
Suppose $\eps>0$ and $n=2,3$. Then the corresponding energy identity \eqref{LLGR:Identity} holds true. Moreover there exist $\delta>0$ and $c>0$ such that for $\|\bs u_0\|_{\dot H^{\frac{n}{2}}} \le \delta$ \begin{equation} \label{eq:L^2_estimate_2}
 \sup_{t >0} \|\bs u(t)-Q\|_{L^2}^2 \le c \|\bs u_0-Q\|_{L^2}^2.
\end{equation}
\end{lemma}

To establish continuity near $t=0$, the Gronwall argument from Lemma \ref{lemma:H^1} may still be invoked to produce an estimate $\|\bs u(t)-Q\|_{H^{\frac n 2}}^2 \le e^{c_0  t} \|\bs u_0-Q\|_{\frac n 2}^2$
with $c_0>0$ only depending on $\|\bs u_0\|_{\dot H^{\frac n 2}}$. 

\begin{proof} The proof of the energy identity is identical to the one in Lemma \ref{lemma:H^1}. 
The proof of \eqref{eq:L^2_estimate_2} anticipates elements from Proposition \ref{proposition:H^{k/2}} below, in particular
to control the homogeneous $H^{\frac n 2}$ seminorm uniformly in time. 
Using the multiplier $\bs u-Q$ for \eqref{LLGR2} we obtain
\[
 \frac{1}{2}  \|\bs u(T)-Q\|_{L^2}^2  + 2 \lambda  \int_0^T  E_{\eps}(\bs u(t)) \, dt =  \frac{1}{2}  \|\bs u_0-Q\|_{L^2}^2 + \lambda  R(T)+  S(T),
\]
where
\[
 R(T)= \frac{1}{2} \int_0^T \scp{\eps (\Delta^2 \bs u\cdot\bs u) - [\mathcal{R}, \bs u] \nabla \bs u}{
 |\bs u-Q|^2}_{L^2} dt
\]
and
\[
S(T)=- \int_0^T \scp{ (\bs u - Q) \times \hl (\bs u -Q)}{  \bs u}_{L^2} \, dt,
\]
since $\scp{\bs u \times \Delta^2 \bs u}{\bs u-Q}_{L^2}= \scp{\bs u \times \Delta \bs u}{\Delta \bs u}_{L^2}+2\scp{\nabla \bs u \times \Delta \bs u}{\nabla \bs u}_{L^2}=0$. \\

Using \eqref{eq:CRW} and \eqref{eq:Sob1} we have
\begin{eqnarray*}
\left| \scp{[\mathcal{R}, \bs u] \nabla \bs u}{ |\bs u-Q|^2}_{L^2} \right| &\lesssim&  \|\bs u\|_{\BMO} \|\nabla \bs u\|_{L^{n}} \|\bs u-Q\|_{L^{\frac{2n}{n-1}}}^2   \, \\
&\lesssim& \|\bs u\|_{\dot{H}^{ \frac{n}{2}}}^2    \|\bs u\|_{\dot{H}^{ \frac{1}{2}}}^2.
\end{eqnarray*}
In view of \eqref{eq:constraint} and integration by parts we have
\begin{eqnarray*}
\scp{\Delta^2 \bs u\cdot \bs u}{\bs u\cdot(\bs u -Q)}_{L^2} &\leq & -2\scp{\Delta\nabla\bs u\cdot\nabla\bs u}{|\bs u - Q|^2}_{L^2}\\
&=& 2\scp{|\Delta\bs u|^2}{|\bs u - Q|^2}_{L^2} + 4 \scp{\Delta\bs u\cdot\nabla\bs u}{\nabla \bs u\cdot Q}_{L^2}.
\end{eqnarray*}
We estimate the terms on the right separately. First,
\[
\left| \scp{\Delta\bs u\cdot\nabla\bs u}{\nabla \bs u\cdot Q}_{L^2} \right| \lesssim  \left|\scp{|\Delta\bs u|}{|\nabla \bs u|^2}_{L^2}\right|   \\
\lesssim  \|\nabla\bs u\|_{L^4}^2\|\Delta \bs u\|_{L^2}
\lesssim  \|\bs u\|_{\dot H^{\frac{n}{2}}}\|\Delta \bs u\|_{L^2}^2,
\]
which follows from Gagliardo-Nirenberg \eqref{eq:GN} and Sobolev embedding.
In addition, \eqref{eq:Sob1} and \eqref{eq:Sob2} imply 
\[
\left| \scp{|\Delta\bs u|^2}{|\bs u - Q|^2}_{L^2}  \right| \lesssim  \|\bs u-Q\|_{L^{\frac{2n}{n-1}}}^2\|\Delta \bs u\|_{L^{2n}}^2
\lesssim  \|\bs u\|_{\dot H^{\frac{1}{2}}}^2\|\Delta \bs u\|_{\dot H^{\frac{n-1}{2}}}^2.
\]
Thus, integrating in time, we infer that
\begin{eqnarray*}
\eps \int_0^T\left| \scp{|\Delta\bs u|^2}{|\bs u - Q|^2}_{L^2}  \right| dt &\leq & \sup_{t\in[0,T]}\|\bs u\|_{\dot H^{\frac{1}{2}}}^2
\int_0^T \eps \|\Delta \bs u\|_{\dot H^{\frac{n-1}{2}}}^2 dt \\
&\lesssim& \|\bs u_0\|_{\dot H^{\frac{1}{2}}}^2\| \bs u_0\|_{\dot H^{\frac{n-1}{2}}}^2 
\lesssim \|\bs u_0 -Q\|_{L^2}^2\| \bs u_0\|_{\dot H^{\frac{n}{2}}}^2,
\end{eqnarray*}
after using \eqref{eq:H^k/2_estimate} for $k=1$ and for $k=n-1$ valid for $\delta$ sufficiently small, and interpolation of homogeneous Sobolev norms. Combining the preceding estimates  we have assuming $\delta \le 1$
$$ R(T) \lesssim  \delta \int_0^T E_{\eps}(\bs u) \, dt+ \delta^2 \|\bs u_0 -Q\|_{L^2}^2. $$
On the other hand, we use  \eqref{eq:Sob1}, \eqref{eq:Sob2} and \eqref{eq:KatoPonce} to bound
\begin{eqnarray*}
\left|\scp{ (\bs u - Q) \times \hl (\bs u -Q)}{  \bs u}_{L^2} \right| &\lesssim& \|   [\hlf, \Omega_{\bs u}]   (\bs u-Q) \|_{L^2} \|\bs u\|_{\dot{H}^{ \frac{1}{2}}} \\ 
&\lesssim& \| \hlf \bs u  \| _{L^{2n}} \|  \bs u -Q \|_{L^{\frac{2n}{n-1}}} \|\bs u\|_{\dot{H}^{ \frac{1}{2}}} \\
&\lesssim& \| \hlf \bs u  \| _{\dot H^{\frac{n-1}{2}}}  \|\bs u\|_{\dot{H}^{ \frac{1}{2}}}^2 \lesssim
\delta  \|\bs u\|_{\dot{H}^{ \frac{1}{2}}}^2.
\end{eqnarray*}
Combining all estimates, the claim follows with an appropriate choice of $\delta$.
\end{proof}

\begin{proposition} \label{proposition:H^{k/2}}
Suppose $\eps>0$, $n=2,3$, and let $k\in\N$ with $1\leq k \leq n+1$. Then there exists $\delta>0$ such that for $\|\bs u_0\|_{\dot H^{\frac{n}{2}}} \le \delta$ and $0 <  T< \infty$
\begin{equation} \label{eq:H^k/2_estimate}
  \|\bs u(T)\|_{\dot{H}^{ \frac k 2}}^2  +  \lambda\int_{0}^T \left( \eps \|\Delta \bs u\|_{\dot{H}^{\frac k 2}}^2+  \|  \bs u\|_{\dot H^ \frac{k+1}{2}}^2 \right) dt \le  \|\bs u_0\|_{\dot H^{\frac k 2}}^2.
\end{equation}
\end{proposition}

\begin{proof}
We focus on the case $k>1$ as the corresponding arguments of Lemma \ref{lemma:H^1}  carry over to the higher dimensional case
under minor modifications.  In order to prove \eqref{eq:H^k/2_estimate} we shall show that for
$2 \le k \le 4$
\begin{equation} \label{eq:k_prelim}
 \left[ \dfrac{1}{2} \|\bs u\|_{\dot H^{\frac k 2}}^2 \right]_0^T + C(T)
  \int_0^T \left(\eps\|\Delta \bs u\|_{\dot H^{\frac k 2}}^2 +  \| \bs u\|_{\dot H^{\frac{k+1}{2}}}^2\right) dt
 \le 0
 \end{equation}
 where for a universal constant $c$
 \[
 C(\tau) = \lambda - c \left((1+\lambda) \sigma(\tau)+\lambda \sigma(\tau)^2 \right)  \quad \text{with} \quad \sigma(\tau) = \sup_{t \in [0,\tau]} \|\bs u\|_{\dot{H}^{\frac{n}{2}}}.
 \]
As in the proof of Lemma \ref{lemma:H^1} we choose $\delta$ to be the positive root of the quadratic equation 
$\lambda = 2c \left((1+\lambda) \sigma+\lambda \sigma^2\right)$ for $\sigma$. Then \eqref{eq:k_prelim} for $k=n \in \{2,3\}$
implies in particular that $\sigma(T) \le \delta$, thus $C(T)\ge \frac{\lambda}{2}$, and the claim follows immediately.

\subsection*{The case k=2}
Using the multiplier $\gl \bs u$ for \eqref{LLGR2}, we obtain 
\begin{eqnarray*}\label{H1estimate}
 \left[ \dfrac{1}{2} \|\bs u\|_{\dot H^1}^2 \right]_0^T + \lambda \int_0^T \left( \| \bs u\|_{\dot H^{\frac 3 2}}^2+\eps\|\Delta \bs u\|_{\dot H^1}^2 \right) dt
 =-\int_0^T\scp{\bs u\times\hl \bs u}{\Delta \bs u}_{L^2} dt
\\
- \int_0^T \lambda \scp{ [\mathcal R, \bs u]\nabla \bs u}{|\nabla \bs u|^2}_{L^2}   +   \eps \left( \lambda\scp{\bs u\cdot\Delta^2 \bs u}{\bs u\cdot\Delta \bs u}_{L^2}+\scp{\bs u\times \Delta^2 \bs u}{\Delta \bs u}_{L^2}\right)  dt.
\end{eqnarray*}
We estimate the terms occuring on the right separately. Let us first consider the projection terms. 
Using  \eqref{eq:CRW} and \eqref{eq:Sob1} we have
\begin{eqnarray}\label{eq:K2Estimate1}
 \left| \scp{ [\mathcal R, \bs u]\nabla \bs u}{|\nabla \bs u|^2}_{L^2}\right|  &\lesssim& \|\bs u\|_{\BMO}\|\nabla \bs u\|_{L^n}\|\nabla \bs u\|_{L^{\frac {2n} {n-1}} }^2 \nonumber \\
&\lesssim&  \|\bs u\|_{\dot H^{\frac{n}{2}}}^2\|\nabla \bs u\|_{\dot H^{\frac{1}{2}}}^2 .
\end{eqnarray}
Moreover
  \begin{equation}\label{eq:K2Estimate2}
  \left|    \scp{\bs u\cdot\Delta^2 \bs u}{\bs u\cdot\Delta \bs u}_{L^2} \right| =
   \left|    \scp{\bs u\cdot\Delta^2 \bs u}{|\nabla \bs u|^2}_{L^2} \right| \lesssim    \| \bs u\|_{\dot{H}^{\frac{n}{2}}}^{2}\|\nabla\Delta \bs u\|_{L^2}^{2}
  \end{equation}
follows after expanding the product according to \eqref{eq:constraint}   \begin{align*}
\||\nabla \bs u|^2\|_{\dot{H}^1}^2 &+ \left| \scp{ |\Delta \bs u|^2}{ |\nabla \bs u|^2}_{L^2} \right|  +\left |\scp{\nabla \bs u\cdot \Delta\nabla \bs u} {|\nabla \bs u|^2}_{L^2}\right | \nonumber \\
 & \lesssim   \|\nabla \bs u\|_{L^{2n}}^2 \|\Delta \bs u\|_{L^{ \frac {2n} {n-1}}}^2 + \|\nabla \bs u\|_{L^{6}}^3\|\nabla\Delta \bs u\|_{L^2} \nonumber \\
&\lesssim    \| \bs u\|_{\dot H^{\frac {n+1} {2}}}^2 \|\Delta \bs u\|_{\dot H^{\frac{1}{2}}}^2 + \|\nabla \bs u\|_{L^{n}}^2\|\nabla \Delta \bs u\|_{L^2}^2 \nonumber \\
&\lesssim   \| \bs u\|_{\dot H^{\frac{n}{2}}}^{\frac {2(5-n)}{6-n}}\|\nabla \Delta \bs u\|_{L^2}^{\frac {2}{6-n}}   \| \bs u\|_{\dot H^{\frac{n}{2}}}^{\frac {2}{6-n}} \|\nabla \Delta \bs u\|_{L^2}^{\frac {2(5-n)}{6-n}} + \| \bs u\|_{\dot{H}^{\frac{n}{2}}}^{2}\|\nabla \Delta \bs u\|_{L^2}^2  
 \end{align*}
using \eqref{eq:Sob1}, \eqref{eq:Sob2}, the Gagliardo-Nirenberg inequality and repeated interpolation.

For the Hamiltonian terms we invoke \eqref{eq:KatoPonce}, \eqref{eq:Sob1} and \eqref{eq:Sob2} in order to deduce
\begin{eqnarray}\label{eq:K2Estimate3}
\left|\scp{\bs u\times\hl \bs u}{\Delta \bs u}_{L^2}  \right|&\lesssim& \|[\hlf \bs u,\Omega_{\bs u}]\nabla \bs u\|_{L^2}\|\nabla \bs u\|_{\dot H^{\frac{1}{2}}} \\
&\lesssim& \|\nabla \bs u\|_{L^{\frac{2n}{n-1}}}\|\hlf \bs u\|_{L^{2n}}\|\nabla \bs u\|_{\dot H^{\frac{1}{2}}} \nonumber\\
&\lesssim& \|\bs u\|_{\dot H^{\frac{n}{2}}}\|\nabla \bs u\|_{\dot H^{\frac{1}{2}}}^2 \nonumber.
\end{eqnarray} 
For the terms in $\eps$, we integrate by parts and use \eqref{eq:Sob1}, \eqref{eq:Sob2} 
and interpolation to obtain 
\begin{eqnarray}\label{eq:K2Estimate4}
\left| \scp{\bs u\times \Delta^2 \bs u}{\Delta \bs u}_{L^2}  \right| &=& \left| \scp{ \nabla \bs u\times \Delta \bs u}{\nabla \Delta \bs u}_{L^2}  \right|\\
&\lesssim& \|\nabla \bs u\|_{L^{2n}} \|\Delta \bs u\|_{L^{ \frac {2n} {n-1}}} \|\nabla \Delta \bs u\|_{L^2} \nonumber \\
& \lesssim &    \| \bs u\|_{\dot H^{\frac {n+1} {2}}} \|\Delta \bs u\|_{\dot H^{\frac{1}{2}}}\|\nabla \Delta \bs u\|_{L^2}\nonumber \\
&\lesssim&    \| \bs u\|_{\dot{H}^{\frac{n}{2}}}\|\nabla\Delta \bs u\|_{L^2}^{2}\nonumber.
\end{eqnarray}

Summarizing \eqref{eq:K2Estimate1}, \eqref{eq:K2Estimate2}, \eqref{eq:K2Estimate3} and \eqref{eq:K2Estimate4} yields \eqref{eq:k_prelim}.

\subsection*{The case k=3}
Using the multiplier $\gl^{\frac{3}{2}} \bs u$ for \eqref{LLGR2} we obtain
\begin{align*}
 \left[ \dfrac{1}{2}  \| \bs u\|_{\dot H^{3/2}}^2 \right]_0^T &+\lambda \int_0^T (\| \bs u\|_{\dot H^2}^2+\eps\|\Delta \bs u\|_{\dot{H}^{3/2}}^2) dt \nonumber 
\\
&=-\lambda\left(\int_0^T   \scp{ [\mathcal R, \bs u]\nabla \bs u}{\bs u\cdot\Delta \hl \bs u}_{L^2}  dt-\eps \scp{\bs u\cdot\Delta^2 \bs u}{\bs u\cdot\Delta \hl \bs u}_{L^2} dt\right) \nonumber  \\
&+\int_0^T  \scp{\bs u\times\hl \bs u}{\Delta \hl \bs u}_{L^2}   dt+\eps\int_0^T   \scp{\bs u\times\Delta^2 \bs u}{\Delta \hl \bs u}_{L^2}   dt.
\end{align*}

We estimate the terms on the right separately. Distributing $\Delta$ we have
\begin{align*}
&  \left|\scp{ [\mathcal R, \bs u]\nabla \bs u}{\bs u\cdot\Delta \hl \bs u}_{L^2}\right| \\
& \lesssim  \left| \scp{[\mathcal R, \bs u]\nabla \bs u}{\nabla \bs u\cdot \nabla\hl \bs u}_{L^2}\right|+ \left|\scp{[\mathcal R, \bs u]\nabla \bs u} {\Delta \bs u\cdot \hl \bs u }_{L^2}\right| + \|\bs u\cdot\hl \bs u\|_{\dot{H}^1}^2.
\end{align*}
taking into account \eqref{eq:commutator_structure} for the last term. Using \eqref{eq:CRW} and
Gagliardo-Nirenberg, it can be estimated by re-introducing a commutator structure similar to \eqref{eq:commutator_structure}
\begin{eqnarray*} 
\|\bs u\cdot\hl \bs u\|_{\dot{H}^1}^2 &\lesssim&  \|\nabla \bs u\cdot \hl \bs u\|_{L^2}^2+\|\bs u\cdot\mathcal R \Delta \bs u\|_{L^{2}}^2\\
&\lesssim& \|\nabla \bs u\|_{L^4}^4+\|\mathcal{R}(\bs u\cdot\Delta \bs u)\|_{L^2}^2+\|[\mathcal R,\bs u]\Delta \bs u\|_{L^2}^2\\
&\lesssim& \|\nabla \bs u\|_{L^4}^4+ \|\bs u\|_{\BMO}^2\|\Delta \bs u\|_{L^2}^2\\
&\lesssim& \|\nabla \bs u\|_{L^n}^2\|\Delta \bs u\|_{L^2}^2.
\end{eqnarray*}
Using \eqref{eq:CRW} and Gagliardo-Nirenberg, the remaining terms can be estimated at once
\begin{align*}
\left| \scp{[\mathcal R,\bs u]\nabla \bs u}{\nabla \bs u\cdot  \nabla\hl \bs u}_{L^2}\right| + \left|\scp{[\mathcal R,\bs u]\nabla \bs u} {\Delta \bs u\cdot \hl \bs u }_{L^2}\right|  \\
\lesssim \|\bs u\|_{\BMO}\|\nabla \bs u\|_{L^4}^2 \|\Delta \bs u\|_{L^2} 
\lesssim   \|\bs u\|_{\dot H^{\frac{n}{2}}}^2\|\Delta \bs u\|_{L^2}^2.
\end{align*}
We obtain
\begin{equation}\label {eq:K3Estimate1}
\left|\scp{ [\mathcal R,\bs u]\nabla \bs u}{\bs u\cdot\Delta \hl \bs u}_{L^2}\right|\lesssim \|\bs u\|_{\dot H^{\frac{n}{2}}}^2\|\Delta \bs u\|_{L^2}^2.
\end{equation}
On the other hand, we claim that the projection terms in $\eps$ satisfy
\begin{equation}\label{eq:K3Estimate2}
\left|\scp{\bs u\cdot\Delta^2 \bs u}{\bs u\cdot\Delta \hl \bs u}_{L^2}\right|\lesssim  \| \bs u\|_{\dot H^{\frac{n}{2}}}\|\Delta \nabla\bs u\|_{\dot H^{\frac{1}{2}}}^2.
\end{equation}
Indeed, taking into account  \eqref{eq:constraint}, the estimate follows from
\begin{align*}\label{eq:K3Estimate2}
 \left|\scp{|\nabla \nabla \bs u|^2}{\bs u\cdot\Delta \hl \bs u}_{L^2}\right|  &+   \left | \scp{\nabla\bs u\cdot\Delta\nabla \bs u}{\bs u\cdot\Delta \hl \bs u}_{L^2}\right|  \nonumber \\ 
&\lesssim   \|\Delta \bs u\|_{L^\frac{4n}{n+1}}^2\|\Delta\nabla  \bs u\|_{L^\frac{2n}{n-1}}+ \|\nabla \bs u \|_{L^n}\|\Delta\nabla  \bs u\|_{L^\frac{2n}{n-1}}^2   \nonumber   \\
&\lesssim \|\nabla \bs u \|_{L^n}\|\Delta\nabla  \bs u\|_{\dot H^{\frac{1}{2}}}^2   \\
&\lesssim  \| \bs u\|_{\dot H^{\frac{n}{2}}}\|\Delta \nabla\bs u\|_{\dot H^{\frac{1}{2}}}^2 \nonumber
\end{align*}
using  \eqref{eq:GN}, \eqref{eq:Sob1} and the Calder\'on-Zygmund inequality.

Finally, we examine the Hamiltonian terms. Integration by parts  yields
\[
\left| \scp{(\bs u\times\hl \bs u)}{\Delta \hl \bs u}_{L^2}   \right| = \left|  \scp{\nabla \bs u\times\hl \bs u}{\nabla \hl \bs u}_{L^2}   \right|,
\]
which is bounded in terms of 
\[
\|\nabla \bs u\|_{L^4}^2\|\Delta \bs u\|_{L^2} \lesssim \|\nabla \bs u\|_{L^n}\|\Delta \bs u\|_{L^2}^2 \lesssim \| \bs u\|_{\dot H^{\frac{n}{2}}}\|\Delta \bs u\|_{L^2}^2.
\]
For the terms in $\eps$ we obtain by using  \eqref{eq:KatoPonce} and \eqref{eq:Sob1}
\begin{align}\label {eq:K3Estimate4}
& \left| \scp {\bs u\times\Delta^2 \bs u}{\Delta \hl \bs u}_{L^2} \right| \\
&\leq\left|  \scp{\hl (\bs u\times\Delta\nabla \bs u)}{\Delta\nabla \bs u}_{L^2}   \right|+ \left| \scp{(\nabla \bs u \times\Delta\nabla \bs u)}{\Delta\hl\bs u}_{L^2}  \right| \nonumber\\
& \lesssim \|[\hlf \bs u,\Omega_{\bs u}]\Delta\nabla \bs u\|_{L^2}\|\Delta\nabla \bs u\|_{\dot H^{\frac{1}{2}}}+\|\nabla \bs u\|_{L^n}\|\Delta\nabla \bs u\|_{L^\frac{2n}{n-1}}^2 \nonumber\\
& \lesssim \|\hlf \bs u\|_{L^{2n}}\|\Delta\nabla \bs u\|_{\dot H^{\frac{1}{2}}}\|\Delta\nabla \bs u\|_{L^{\frac{2n}{n-1}}}+ \|\bs u\|_{\dot H^{\frac{n}{2}}} \|\Delta\nabla \bs u\|_{\dot H^{\frac{1}{2}}}^2 \nonumber \\
& \lesssim \|\bs u\|_{\dot H^{\frac{n}{2}}} \|\Delta\nabla \bs u\|_{\dot H^{\frac{1}{2}}}^2 \nonumber.
\end{align}

Thus, after collecting  estimates  \eqref{eq:K3Estimate1},  \eqref{eq:K3Estimate2},  \eqref{eq:K3Estimate4} together, we deduce 
\eqref{eq:k_prelim}.

\subsection*{The case k=4}
Using the multiplier $\Delta^2 \bs u$ for \eqref{LLGR2}, we obtain
 \begin{align}\label{H2Estimates}
\left[\dfrac{1}{2}  \| \bs u\|_{\dot H^2}^2\right]_0^T &+ \lambda \int_0^T(\| \bs u\|_{\dot H^{5/2}}^2+\eps\|\Delta \bs u\|_{\dot H^2}^2)dt \nonumber \\
&=-\lambda \int_0^T \left(  \scp{[\mathcal R,\bs u]\nabla \bs u}{\bs u\cdot\Delta^2 \bs u}_{L^2}-\eps \| \bs u\cdot \Delta^2 \bs u \|_{L^2}^2 \right)dt \nonumber \\
&+\int_0^T  \scp{\bs u\times\hl \bs u}{\Delta^2 \bs u}_{L^2}  dt.
\end{align}
For the first term on the right, we deduce from \eqref{eq:constraint}
\[
 \left|\scp{[\mathcal R,\bs u]\nabla \bs u}{\bs u\cdot\Delta^2 \bs u}_{L^2}\right|  \lesssim  \left|\scp {| [\mathcal R,\bs u]\nabla \bs u|}{|\nabla \nabla  \bs u|^2 }_{L^2}\right|  +  \left|\scp {[\mathcal R,\bs u]\nabla \bs u} {(\Delta\nabla \bs u\cdot \nabla \bs u)}_{L^2}\right|,
 \]
where 
   \begin{eqnarray*}
 \left|\scp {|[\mathcal R,\bs u]\nabla \bs u|}{|\nabla \nabla \bs u|^2 }_{L^2}\right|  &\lesssim& \|\Delta \bs u\|_{L^{\frac{2n}{n-1}}}^2\|\bs u\|_{\BMO}\|\nabla \bs u\|_{L^n}\\
 &\lesssim& \|\Delta \bs u\|_{\dot H^{\frac{1}{2}}}^2\|\bs u\|_{\dot{H}^{\frac{n}{2}}}^2 \nonumber
 \end{eqnarray*} 
by \eqref{eq:CRW}, the Calder\'on-Zygmund inequality and \eqref{eq:Sob1}.
Moreover, integrating by parts and taking into account
$
\nabla [\mathcal{R}, \bs u] \nabla \bs u = [\mathcal{R}, \bs u] (\nabla \nabla \bs u)+ [\mathcal{R}, \nabla \bs u] \nabla \bs u
$
the remaining term is bounded in terms of the following three estimates. First,
 \begin{eqnarray*}
  \left|\scp {[\mathcal R,\bs u](\nabla \nabla \bs u)} {(\nabla \nabla \bs u\cdot \nabla \bs u)}_{L^2}\right|
&\lesssim& \|\bs u\|_{\BMO}\|\nabla \bs u\|_{L^n} \|\Delta \bs u\|_{L^{\frac{2n}{n-1}}}^2\\
&\lesssim& \|\bs u\|_{\dot{H}^{\frac{n}{2}}}^2  \|\Delta \bs u\|_{\dot H^{\frac{1}{2}}}^2
\end{eqnarray*}
using \eqref{eq:CRW}, the Calder\'on-Zygmund inequality and \eqref{eq:Sob1}. In addition,
\begin{eqnarray*}
  \left|\scp {[\mathcal R, \nabla \bs u] \nabla \bs u} {(\nabla \nabla \bs u\cdot \nabla \bs u)}_{L^2}\right|
&\lesssim& \|\nabla \bs u\|_{L^n}^3    \|\nabla \nabla  \bs u\|_{L^{\frac{2n}{n-1}}} \\
&\lesssim& \|\bs u\|_{\dot{H}^{\frac{n}{2}}}^2 \|\Delta \bs u\|_{\dot H^{\frac{1}{2}}}^2 
\end{eqnarray*}
follows by splitting up the commutator and using the Gagliardo-Nirenberg inequality
$$\|\nabla \bs u\|_{L^{\frac{6n}{n+1}}}^3\lesssim\|\nabla \bs u\|_{L^n}^2\|\Delta \bs u\|_{L^{\frac{2n}{n-1}}},$$
combined with \eqref{eq:Sob1}. Finally,
 \begin{eqnarray*}
   \left| \scp{[\mathcal R,\bs u]\nabla \bs u}{|\nabla \nabla \bs u|^2}_{L^2}\right| 
&\lesssim&  \|\bs u\|_{\BMO}\|\nabla \bs u\|_{L^n} \|\Delta \bs u\|_{L^{\frac{2n}{n-1}}}^2\\
&\lesssim&  \|\bs u\|_{\dot{H}^{\frac{n}{2}}}^2 \|\Delta \bs u\|_{\dot H^{\frac{1}{2}}}^2,
\end{eqnarray*}
via \eqref{eq:CRW} and \eqref{eq:Sob1}.
Hence we obtain
\begin{equation}\label{eq:K4Estimate1}
 \left|\scp{[\mathcal R,\bs u]\nabla \bs u}{\bs u\cdot\Delta^2 \bs u}_{L^2}\right| \lesssim \|\bs u\|_{\dot{H}^{\frac{n}{2}}}^2
 \|\Delta \bs u\|_{\dot H^{\frac{1}{2}}}^2.
\end{equation}
As far as terms in $\eps$ are concerned, the following estimate holds true
\begin{equation}\label {eq:K4Estimate2}
 \| \bs u\cdot \Delta^2 \bs u \|_{L^2}^2  \lesssim \|\bs u\|_{\dot{H}^{\frac{n}{2}}}^2\|\Delta^2 \bs u\|_{L^2}^2.
 \end{equation}
 Indeed, due to  \eqref{eq:constraint} and the Calder\'on-Zygmund inequality, it suffices to estimate
\begin{align*}
 \|\Delta \bs u\|_{L^4}^4 &+ \|\nabla \bs u\|_{L^4}^2  \|\nabla\Delta \bs u\|_{L^4}^2  \nonumber \\
 &  \lesssim    \|\Delta \bs u \|_{L^n}^2\|\nabla\Delta \bs u\|_{L^2}^2+\|\nabla \bs u\|_{L^n}  \|\Delta \bs u\|_{L^2} \|\nabla \Delta \bs u\|_{L^n}  \|\Delta^2 \bs u\|_{L^2}  \nonumber \\
  &  \lesssim    \|\nabla \bs u\|_{\dot{H}^{\frac{n}{2}}}^2\|\nabla\Delta \bs u\|_{L^2}^2+\|\bs u\|_{\dot H^{\frac{n}{2}}}  \|\Delta \bs u\|_{L^2} \| \Delta \bs u\|_{\dot H^{\frac{n}{2}}}  \|\Delta^2 \bs u\|_{L^2}  \nonumber \\
  &  \lesssim  \|  \bs u\|_{\dot H^{\frac{n}{2}}}^{\frac{12-2n}{8-n}} \|\Delta^2 \bs u\|_{L^2}^{\frac{4}{8-n}} \| \bs u\|_{\dot H^{\frac{n}{2}}}^{\frac{4}{8-n}} \|\Delta^2 \bs u\|_{L^2}^{\frac{12-2n}{8-n}}+ \|\bs u\|_{\dot{H}^{\frac{n}{2}}}^2\|\Delta^2 \bs u\|_{L^2}^2  \nonumber 
 \end{align*}
using the H\"older inequality, \eqref{eq:GN} and interpolation of homogeneous Sobolev norms.
Finally, we claim that the  Hamiltonian terms satisfy 
\begin{equation} \label{eq:K4Estimate3}
 \left| \scp{(\bs u\times\hl \bs u)}{\Delta^2 \bs u}_{L^2}\right|  \lesssim  \| \bs u\|_{\dot H^{\frac{n}{2}}}\|\Delta \bs u\|_{\dot H^{\frac{1}{2}}}^2.
 \end{equation}
Indeed, integrating by parts and using \eqref{eq:KatoPonce} and \eqref{eq:constraint} provides the estimate
\begin{align*}
&\left| \scp{\nabla \bs u\times\hl \nabla \bs u}{\Delta \bs u}_{L^2} \right|+\left| \scp{\bs u\times\hl \Delta \bs u}{\Delta \bs u}_{L^2}\right| \nonumber\\
& \lesssim  \|\nabla \bs u\|_{L^n}\|\Delta \bs u\|_{L^{\frac{2n}{n-1}}}^2 + \scp{ [\hlf \bs u,\Omega_{\bs u}]\hlf\Delta u}{ \Delta u}_{L^2}  \nonumber \\
& \lesssim  \| \bs u\|_{\dot H^{\frac{n}{2}}}\|\Delta \bs u\|_{\dot H^{\frac{1}{2}}}^2 +\|\hlf \bs u\|_{L^{2n}} \|\hlf\Delta \bs u\|_{L^2}\|\Delta \bs u\|_{L^{\frac{2n}{n-1}}}  \\
& \lesssim  \| \bs u\|_{\dot H^{\frac{n}{2}}}\|\Delta \bs u\|_{\dot H^{\frac{1}{2}}}^2 \nonumber.
\end{align*}
Summarizing \eqref{eq:K4Estimate1}, \eqref{eq:K4Estimate2} and \eqref{eq:K4Estimate3} implies \eqref{eq:k_prelim}.
\end{proof}

\section{Proof of the main result}
Given initial data $\bs u_0 \in H^{\frac{n}{2}}_Q(\R^n; \St)$ there exists, by virtue of a Schoen-Uhlenbeck argument as e.g. in \cite{Melcher} Lemma 4, an approximating family $\bs u^{(\eps)}_0 \in H^\infty_Q(\R^n;\St)$ such that 
\[
\bs u^{(\eps)}_0 - \bs u_0 \to 0 \quad  \text{in} \quad H^{\frac{n}{2}}(\R^n; \R^3) \quad \text{and} \quad E_\eps(\bs u^{(\eps)}_0) \to E_0(\bs u_0)=  \frac 1 2 \|\bs u_0\|_{\dot{H}^{\frac{1}{2}}}^2
\]
as $\eps \to 0$. By Lemma \ref{lemma:H^1} and Proposition \ref{proposition:H^{k/2}},  
the corresponding solutions $\bs u^{(\eps)}$ subconverge weakly to  an element 
\[
\bs u \in L^\infty((0,\infty); H^{\frac{n}{2}}_Q(\R^n;\St)) 
\quad \text{with} \quad
\nabla \bs u, \del_t \bs u \in  L^2((0,\infty);H^{\frac{n-1}{2}}(\R^n))
\] 
in the corresponding weak topology solving \eqref{eq:HLLG}. Indeed, the passing to the limit $\eps \to 0$ in \eqref{eq:LLGR} is customarily carried out on the basis of the Gilbert form \eqref{eq:Gilbert}, i.e., 
\[
\beta  \del_t \bs u^{(\eps)}  -\bs u^{(\eps)}\times   \left( \alpha  \del_t \bs u^{(\eps)} +   \hl \bs u^{(\eps)} \right)= - \eps \nabla \cdot \left( \bs u^{(\eps)}  \times \nabla \bs u^{(\eps)} \right) 
 \]
for $n=1$, and factorizing the bi-Laplacian in $\bs u \times \Delta^2 \bs u$ 
\begin{align*}
\beta  \del_t \bs u^{(\eps)}  -\bs u^{(\eps)}\times &  \left( \alpha  \del_t \bs u^{(\eps)} +   \hl \bs u^{(\eps)} \right)
\\ & =  \eps  \left[ \Delta \left( \bs u^{(\eps)}  \times \Delta \bs u^{(\eps)} \right) - 2 \nabla \cdot   \left( \nabla \bs u^{(\eps)}  \times \Delta  \bs u^{(\eps)}\right)    \right]
\end{align*} 
for $n=2,3$. In view of the energy inequality, the $\eps$ terms on the right are vanishing in the sense of distributions as $\eps \to 0$, while the terms on the left are weakly continuous on bounded sets by virtue of the Rellich-Kondrachov compactness theorem. In the case of heat flows to arbitrary spheres, the vector products are to be replaced by outer products, see e.g. \cite{Helein:book}, while the compactness arguments remain the same.

Initial data $\bs u_0$ is strongly attained in $H^{\frac{n}{2}}$ as $t \searrow 0$ as a consequence of weak attainment and convergence of norms
\[
\limsup_{t \searrow 0} \|\bs u(t)-Q\|_{H^{\frac{n}{2}}} \le  \|\bs u_0-Q\|_{H^{\frac{n}{2}}}
\]
following from \eqref{eq:L^2_estimate} in Lemma \ref{lemma:H^1} and its higher dimensional version (see remark after
Lemma \ref{lemma:L^2_higher}), respectively.

According to the regularity of $\nabla \bs u$ we also have 
$\bs u(t) \in H^{\frac{n+1}{2}}_Q$ for times $t>0$ which are arbitrarily small. Hence, by 
Proposition \ref{proposition:H^{k/2}} applied after translation in time and Proposition \ref{prop:uniqueness}, we have
\[
\bs u \in L^\infty_{\rm loc}((0, \infty); H^{\frac{n+1}{2}}_Q(\R^n)).
\]
Interpolating this with 
\[
\bs u \in H^1_{\rm loc}((0,\infty); L^2_Q(\R^n)) \subset C^{\frac{1}{2}}_{\rm loc}((0, \infty); L^2_Q(\R^n))
\]
from Lemma \ref{lemma:H^1}, we obtain
\[
\bs u \in C^{\frac{1-\theta}{2}}_{\rm loc}( (0,\infty); H^{\frac{\theta(n+1)}{2}}_Q(\R^n)) \quad \text{for every}  \quad 0<\theta< 1.
\]
For suitable choices of $\theta$ while using the Sobolev embedding theorem we have in particular
\[
\bs u\in C^{\frac{1}{2(n+1)}}_{\rm loc}( (0,\infty); H^{\frac{n}{2}}_Q(\R^n)) \cap C^{\frac{1}{n+2}}_{\rm loc}( \R^n \times (0,\infty)), 
\]
and from the continuity at $t=0$
\[
\bs u \in C^0([0, \infty); H^{\frac{n}{2}}_Q(\R^n)).
\]

Finally the long-time asymptotics in the supercritical case $n>1$ follows from Agmon's inequality
\[
\|\bs u(t)-Q\|_{L^\infty} \le 2 \|\bs u(t)-Q\|_{L^2}^{\frac{1}{n+1}}   \|\nabla \bs u(t) \|_{\dot H^{\frac{n-1}{2}}}^{\frac{n}{n+1}},
\]
taking into account Lemma \ref{lemma:L^2_higher} and the decay of $\| \nabla \bs u(t) \|_{\dot H^\frac{n-1}{2}}$ as $t \to \infty$.\\
The latter one follows by using multipliers $t \gl^{\frac{n+1}{2}} \bs u(t)$ rather than $\gl^{\frac{n+1}{2}} \bs u$ in the proof of the corresponding estimates in Lemma \ref{lemma:H^{3/2}} and Proposition \ref{proposition:H^{k/2}} yielding by the same line of arguments for $T_0$ positive
\[
\int_{T_0}^T t \del_t \| \nabla \bs u(t)\|_{\dot{H}^{\frac{n-1}{2}}}^2 \, dt \le 0 \quad \text{hence} \quad
\left[ \| \nabla \bs u\|_{\dot{H}^{\frac{n-1}{2}}}^2  \right]_{T_0}^T \le \dashint_{T_0}^T  \| \nabla \bs u(t)\|_{\dot{H}^{\frac{n-1}{2}}}^2 \, dt,
\] 
which in view of the time integrated bounds immediately implies the claim. 

\subsection*{Acknowledgements} 
This work is partially supported by Deutsche Forschungsgemeinschaft (DFG grant no. ME 2273/3-1). 

\setcounter{equation}{0}
\renewcommand\theequation{A.\arabic{equation}}
\setcounter{section}{0}

\appendix
 \section{}\label{appendix}
For the convenience of the reader, we collect some interpolation inequalities that are used frequently in this work. We begin with the standard interpolation inequality for $\dot H^s$ seminorms
\[
\|f\|_{\dot H^{s}}\lesssim \|f\|_{\dot H^{s_1}}^\theta \|f\|_{\dot H^{s_2}}^{(1-\theta)},
\]
valid for $s_1<s<s_2$, with $s=\theta s_1+(1-\theta)s_2$ which readily follows from its characterization in Fourier space. An endpoint version is Agmon's inequality
\[
\| f \|_{L^\infty}\lesssim \|f\|_{\dot H^{s_1}}^\theta \|f\|_{\dot H^{s_2}}^{(1-\theta)}, 
\]
 holding true for $0\leq s_1<\dfrac{n}{2}<s_2$, with $\dfrac{n}{2}=\theta s_1+(1-\theta)s_2$.

Moreover, we point out the following instance of the
Gagliardo-Nirenberg inequality
\begin{equation}\label{eq:GN}
 \|\nabla^{l+1} f\|_{L^p}\lesssim \|\nabla f \|_{L^n}^{1-\theta}\|\nabla^{m+1} f\|_{L^q}^{\theta},  
 \end{equation}
where
$\dfrac{1}{p}=\dfrac{l}{n} + \theta\left(\dfrac{1}{q} - \dfrac{m}{n}  \right) + \dfrac{(1-\theta)}{n}$ and $\dfrac{l}{m}\leq \theta\leq 1$. \\
The previous relation proves particularly useful in our context, in view of the Sobolev inequality 
\[
\|\nabla f\|_{L^n} \lesssim \| f\|_{\dot{H}^{ \frac{n}{2}}},
\]
which  we repeatedly use  along with
  \begin{equation} \label{eq:Sob1}
\| f\|_{L^{\frac{2n}{n-1}}}  \lesssim \|f\|_{\dot H^{\frac{1}{2}}}
\end{equation}
and 
\begin{equation} \label{eq:Sob2}
\|f\|_{L^{2n}} \lesssim  \|f\|_{\dot H^{\frac{n-1}{2}}}
\end{equation}
for appropriate functions $f:\R^n \to \R$. In turn, these estimates are frequently used in conjunction with the following Gagliardo-Nirenberg type inequality
\begin{equation} \label{eq:GN2}
\|f\|_{L^{4n}}^2 \lesssim \| f \|_{L^{2n}}  \|\nabla f\|_{L^n}.
\end{equation}

Finally,  a fractional Gagliardo-Nirenberg estimate is needed, for which we present a simple proof.
\begin{lemma}\label{lemma:interpolation}
For $f\in H^{\frac{1}{2}}(\R^n)$ the following estimate holds true
\begin{equation}\label{eq:GN1}
\|f\|_{L^{\frac{4n}{2n-1}}}^4 \lesssim \| f \|_{L^2}^2 \|f\|_{\dot{H}^{\frac{1}{2}}}^2.
\end{equation}
\end{lemma}

\begin{proof}
Applying the Hausdorff-Young inequality, we have
\[
\|f\|_{L^{\frac{4n}{2n-1}}}^4 \lesssim \|\hat{f}\|_{L^{\frac{4n}{2n+1}}}^4 
\lesssim \left( \int_{|\xi|\le 1} |\hat{f}(\xi)|^{{\frac{4n}{2n+1}}} d \xi \right)^\frac{2n+1}{n} + \left( \int_{|\xi|\ge 1} |\hat{f}(\xi)|^{{\frac{4n}{2n+1}}}d \xi \right)^\frac{2n+1}{n}.
\]
The first integral can be estimated by virtue of Jensen's inequality and Parseval's identity by $\|f\|_{L^2}^4$. As for the second integral we use H\"older's inequality
\[
 \int_{|\xi|\ge 1} |\hat{f}(\xi)|^\frac{4n}{2n+1} d \xi \le \left( \int_{|\xi|\ge 1} |\xi| |\hat{f}(\xi)|^{2} d \xi \right)^\frac{2n}{2n+1}
  \left( \int_{|\xi|\ge 1} \frac{d \xi}{|\xi|^{2n}} \right)^\frac{1}{2n+1}. 
\]
It follows that $\|f\|_{L^\frac{4n}{2n-1}}^4 \lesssim \| f \|_{L^2}^4+ \|f\|_{\dot{H}^{\frac{1}{2}}}^4$ and by scaling 
\[
 \lambda^{n/4}  \|f\|_{L^4}^4 \lesssim \lambda^{\frac{n}{2}} \| f \|_{L^2}^4+ \lambda^{\frac{n-1}{2}} \|f\|_{\dot{H}^{\frac{1}{2}}}^4
 \quad \text{for every} \quad  \lambda >0.
 \]
The claim then follows after optimising in $\lambda$.
 \end{proof}
\bibliography{MelSak_CPDE}
\bibliographystyle{acm}
\end{document}